\documentclass{article}

\usepackage{fullpage}
\usepackage{latexsym}
\usepackage{graphicx}
\usepackage{mathptmx}
\usepackage{authblk}

\usepackage[style=authoryear, backend=biber]{biblatex}

\usepackage{amsmath}
\usepackage{amsfonts}
\usepackage{amssymb}
\usepackage{amsbsy}
\usepackage{amsthm}
\usepackage{bm}
\usepackage{multirow}

\newcommand{\citeN}[1]{\textcite{#1}}
\renewcommand{\cite}[1]{\parencite{#1}}

\newtheorem{theorem}{Theorem}

\newtheorem{remark}{Remark}

\newcommand{\EE}{ {\mathbb{E } } }

\newcommand{\target}{ {k } }

\usepackage[ruled,vlined]{algorithm2e}

\begin{document}

\title{Tail Quantile Estimation for Non-preemptive Priority Queues }

\author[12]{Jin Guang}
\author[2]{Guiyu Hong}
\author[2]{Xinyun Chen}
\author[3]{Xi Peng}
\author[3]{Li Chen}
\author[3]{Bo Bai}
\author[3]{Gong Zhang}
\affil[1]{Shenzhen Research Institute of Big Data, Guangdong 518172, P. R. CHINA}
\affil[2]{School of Data Science, The Chinese University of Hong Kong, Shenzhen, Guangdong 518172, P. R. CHINA}
\affil[3]{Theory Lab, Central Research Institute, 2012 Labs, Huawei Technologies, Co., Ltd., Hong Kong, P. R. CHINA}

\date{}

\maketitle

\begin{abstract}
    Motivated by applications in computing and telecommunication systems, we investigate the problem of estimating $p$-quantile of steady-state sojourn times in  a single-server multi-class queueing system with non-preemptive priorities for $p$ close to 1. The main challenge in this problem lies in efficient sampling from the tail event. To address this issue, we develop  a regenerative simulation algorithm with importance sampling.  In addition, we establish a central limit theorem for the estimator to construct the confidence interval. Numerical experiments show that our algorithm outperforms benchmark simulation methods. Our result contributes to the literature on rare event simulation for queueing systems. 
\end{abstract}

\section{INTRODUCTION}\label{sec:intro}
In computing and telecommunication industries, service-level agreement (SLA)  usually takes the form of a guarantee on the tail probability of response times,
$$P\left(\text{response time}>\gamma\right)<1- p.$$
For example, $p=0.99$ and $\gamma = 400$ms  \cite{harchol2021open}. In practice, the tail probability $1-p$ is usually fixed, and  the service provider needs a good estimator of $p$-quantile of response time before drafting an SLA contract. A computing or telecommunication system usually involves multiple priority classes and protocols, for which an analytic solution of the response time distribution is unavailable.  In this paper, we investigate simulation techniques to estimate the $p$-quantile of steady-state response time for a given $p\approx 1$.

In particular, we consider a single node with priority queueing (PQ) scheduling as in Huawei CloudEngine 12800 and 12800E \cite{huawei}. Mathematically, it can be modeled  as a non-preemptive single server with  multiple priority classes. Suppose $F$ is the distribution function of steady-state response time. Then the $p$-quantile of $F$, $0<p<1$, is defined to be $Q(p)=\inf\{x:F(x)\ge p\}$.  We develop a simulation algorithm that efficiently estimates the quantile $Q(p)$ for each priority class. To estimate the steady-state distribution and facilitate output analysis, we apply the regenerative simulation framework \cite{asmussen2007stochastic}. A good estimator for the quantile $Q(p)$ essentially relies on sufficient samples from the rare event $\{\text{response time}\approx Q(p)\}$ for $p\approx 1$. For this purpose, we apply importance sampling (IS) method that is widely used in rare event simulation and apply the cross-entropy (CE) method \cite{de2004fast} to optimize the IS distribution. {Based on the regenerative cycles generated under the importance distribution, we design a new quantile estimator and establish the corresponding central limit theorem.  As a consequence, we can construct confidence intervals for the quantile using batching methods.} The numerical results show that our algorithm outperforms  benchmark algorithms. Finally, we apply our algorithm to SLA estimation for a node with 8 priority classes.

Our work is related to the literature on rare event simulation of queueing systems, see \citeN{blanchet2009rare} and the references therein. Most existing works are on single-class systems, including G/G/1 queue \cite{Blanchet2007GG1}, fluid network \cite{chang1994effective} and Jackson network \cite{dupuis2009importance}.  There are a few exceptions on processing sharing system \cite{mandjes2006large} and priority queues \cite{setayeshgar2012large}. In those works, the IS distributions are typically derived via large deviation principle (LDP),  see for example \citeN{parekh1989quick} or \citeN{dupuis2007dynamic}. Among these works, the setting in  \citeN{setayeshgar2012large} is probably the closest to ours. In \citeN{setayeshgar2012large}, the author studies IS for a  preemptive server with 2 priority classes based on LDP analysis. To the best of our knowledge, however, there is no existing result on non-preemptive priority queues. This is probably because both LDP and differential game approaches rely on sophisticated analysis on the system dynamics and hence are hard to be applied to systems with multiple customer classes and general service protocols. In this light, we apply the CE method to  optimize the IS distribution numerically. In our numerical experiments, we show that preemptive and non-preemptive queues behave differently conditional on the rare event. As a consequence, our algorithm outperforms the IS algorithm proposed in \citeN{setayeshgar2012large}, which was designed for preemptive systems. We also note that most existing works {in the area of rare event simulation} have focused on the distribution of total workload in the system, while our algorithm works for the sojourn time of individual customers. 

\textbf{Organization of the rest of this paper:} The queueing model and the simulation goal are explained in Section \ref{sec: model}. In Section \ref{sec: algorithm}, we first explain the key components in algorithm design and then present the complete algorithm. In Section \ref{sec: CI}, we establish CLT for our quantile estimator  and construct the confidence intervals. Numerical results  are reported  in Section \ref{sec: numerical}.

\section{PROBLEM SETTING}\label{sec: model}

We consider a single server with jobs of multiple priority classes. Let $K$ be the number of priority classes. We index the classes by $k=1,2,...,K$, with class $1$ having the highest priority and class $K$ the lowest priority. For $k=1,2,...,K$, jobs of class $k$ arrive according to a Poisson process with the rate $\lambda_k$ and have i.i.d. exponential service times with mean $1/\mu_k$. Following the setting in real systems, we assume that the server is \textbf{non-preemptive}, i.e.,
low priority jobs in service are not interrupted by high priority jobs. 
The arrival processes and service times are assumed to be mutually independent. 

Our goal is to estimate a commonly used performance metric in SLA \cite{harchol2021open} for the above non-preemptive system via simulation. In detail, for  a given $p\in(0,1)$ and each class $k$, let $Q_k(p)$ be the  $p$-tile of total sojourn time, also known as response time, $R_{k,\infty}$  of class-$k$ jobs in the \textbf{steady state}, i.e.,
\begin{equation}\label{eq:quantile}
    Q_k(p)=\inf\{\gamma:P(R_{k,\infty}\le \gamma)\ge p\}=\inf\{\gamma:P(R_{k,\infty}>\gamma)<1-p\}.
\end{equation}
For a given $p\in(0,1)$, we need to estimate $Q_k(p)$ for all $k=1,2,...,K$. In the setting of SLA, $p$ is typically close to 1, say $p=0.999$ or 0.99999. 
In other words, the problem is essentially to estimate the quantile corresponding to an extreme event in the steady-state distribution of a queueing system.

\section{ALGORITHM DESIGN}\label{sec: algorithm}
For  a given priority class $k$, Let $F(\gamma) = P(R_{k,\infty}\leq \gamma)$ be the CDF of steady-state response time. 
Intuitively, we can obtain a good estimation of $Q(p)$ if we have a good estimation of $F(\gamma)$ for $\gamma$ in a neighborhood around $Q(p)$. Our algorithm design is based on this idea and contains three key components. First, to deal with the steady-state distribution, we apply the regenerative simulation technique (Section \ref{subsec: regenerative}). As $\{R_\infty>\gamma\}$ are rare event for $\gamma$ close to $Q(p)$, we use IS to improve simulation efficiency (Section \ref{subsec: IS}). The IS distribution is optimized by cross-entropy method (Section \ref{subsec: CE}). Finally, we present our estimator for $Q(p)$  in Section \ref{subsec: complete algorithm} along with the complete algorithm.

\subsection{Regenerative Simulation}\label{subsec: regenerative}

The dynamics of the queueing model described in Section \ref{sec: model} can be viewed as a regenerative process. In particular, the system  regenerates whenever a job finds the server idle upon departure.
For a regenerative cycle, denote by $\alpha$ the cycle length in unit of time and $\alpha_k$ the total number of jobs of class $k$ served in this cycle. For $n=1,2,...,\alpha_k$,
let $R_{k,n}$ be the response time  of the $n$-th job of class $k$. 
Then by renewal reward theorem \cite{crane1977introduction}, we have, for any priority class $k$ and $\gamma>0$
\begin{equation}\label{eq:RM}
	P\left(R_{k,\infty}>\gamma\right)
	=\frac{\EE\left[\sum_{n=1}^{\alpha_k} 1_{\{R_{k,n}>\gamma\}}\right]}{\EE\left[\alpha_k\right]}.
\end{equation}

The enumerator $\EE\left[\sum_{n=1}^{\alpha_k} 1_{\{R_{k,n}>\gamma\}}\right]$ involves a rare event when $\gamma$ is large and hence will be estimated by IS method which we explain in the Section \ref{subsec: IS}. But to estimate the denominator $\EE[\alpha_k]$ does not necessarily requires importance sampling. 

\subsection{Importance Sampling for Tail Probability Estimation}\label{subsec: IS}
Intuitively, in order to get good estimators for $Q_k(p)$,  we need to choose the importance distribution such that the events $\{R_{k,n}>\gamma\}$ happens frequently for $\gamma\approx Q(p)$. Suppose $Q(p)<\gamma_{\max}$ and the constant $\gamma_{\max}$ is known. Then, we shall design the importance distribution such that event $\{R_{k,n}>\gamma_{\max}\}$ happens frequently and hence so are $\{R_{k,n}>\gamma\}$ for $\gamma\approx Q_k(p)$.

The queueing system is driven by the jobs' inter-arrival times and services times. Denote by $f^A_l(\cdot)$ and $f^S_l(\cdot)$ be the density of inter-arrival times and service times of class-$l$ jobs, for $l=1,2,...,K$. 
Our design of IS simulation follows the so-called switching change of measure \cite{rubinstein2004cross}. In particular, starting from time $0$, we generate the inter-arrival times and service times according to some  IS density $\tilde{f}^A_l$ and $\tilde{f}^S_l$, $l=1,2,...,K$ and simulate the system dynamics accordingly. The choice of $\tilde{f}^A_l$ and $\tilde{f}^S_l$ will be discussed in Section \ref{subsec: CE}. So for, let's take $\tilde{f}^A_l$ and $\tilde{f}^S_l$ as given.  In switching change of measure, the simulator will switch back to the original distribution  $f^A_l$ and $f^S_l$  once the event of interest $\{R_{k,n}>\gamma_{\max}\}$ is observed. In principle, the event $\{R_{k,n}>\gamma_{\max}\}$ is observed only after job $n$ departs. To  further reduce the variance caused by the random likelihood ratio, we shall define a more sophisticated switching rule as follows.

Suppose for jobs of classes $l=1,2,..., k$, their service times are realized upon arrival. At the arrival time of  job $n$ in class $k$, define $R'_{k,n}$ as the sum of service times of all jobs of class $1\leq l\leq k$ waiting in queues and the remaining service time of the job in service. Note that  $R'_{k,n}$ is a lower bound of the actual response time $R_{k,n}$ and is known immediately upon arrival. Define $\tau_k$ as the arrival time of the first class-$k$ job such that $R^{\prime}_{k,n}>\gamma_{\max}$. 
\begin{equation*}
	\tau_k= \inf \left\{ \sum_{i=1}^nA_{k,i}: R^{\prime}_{k,n}>\gamma_{\max} \right\}.
\end{equation*}
By definition, $\tau_k$ is a stopping time with respect to filtration $\{\mathcal{F}_t:t\geq 0\}$, where $\mathcal{F}_t$ includes events corresponding to all arrivals before time $t$, service times of jobs in classes $l\leq k$ that have arrived by time $t$, and service times of jobs in classes $l>k$ that have entered service by time $t$. In our algorithm, the simulator will switch back to the original distribution after $\tau_k\wedge \alpha$ and terminate at $\alpha$ when a job finds the server idle upon departure. 

Now we explain how to compute the likelihood ratio function for a sample path of the regenerative cycle generated by our IS algorithm.  For each class $l=1,2,...,K$, denote by $A_{l,n}$ the epoch between the $(n-1)$-th and $n$-th arrivals of class $l$, and denote by $S_{l,n}$ the service time of $n$-th arrival of class $l$. Let $G^A_l(\cdot)$ be the tail probability of inter-arrival times in class $l$. 
For any time $t\in[0,\alpha]$, denote by $N_l^A(t)$ and $N_l^S(t)$  the number of job arrivals  and  jobs that have entered service of class $l$ by time $t$, respectively. Then, the likelihood ratio upon time $t$ can be computed as
\begin{equation}\label{eq: LR}
	L(t) = \frac{\prod_{l=1}^K\left(\prod_{n=1}^{N^A_l(t\wedge \tau_k)}f^A_l(A_{l,n})\cdot G^A_l(H^A_l)\right)\cdot \prod_{l=1}^k\prod_{n=1}^{N^A_l(t\wedge \tau_k)}f^S_{l}(S_{l,n})\cdot\prod_{l=k+1}^K\prod_{n=1}^{N_l^S(t\wedge \tau_k)}f^S_{l}(S_{l,n}) }{\prod_{l=1}^K\left(\prod_{n=1}^{N_l^A(t\wedge \tau_k)}\tilde{f}^A_l(A_{l,n})\cdot \tilde{G}^A_l(H^A_l)\right)\cdot \prod_{l=1}^k\prod_{n=1}^{N_l^A(t\wedge \tau_k)}\tilde{f}^S_{l}(S_{l,n})\cdot\prod_{l=k+1}^K\prod_{n=1}^{N_l^S(t\wedge \tau_k)}\tilde{f}^S_{l}(S_{l,n})},
\end{equation}
where $H^A_l$ is the age since last arrival of class $l$ at time $t\wedge \tau_k$. Note that the above formula of likelihood still holds if $t$ is replaced with a stopping time with respect to the filtration $\{\mathcal{F}_t\}$, see Chapter 5.1c of \citeN{asmussen2007stochastic}.

Let $E_{k,n}$ be the enter-service time of the $n$-th job in class $k$. Then, $R_{k,n}\in\mathcal{F}_{E_{k,n}}$. As a consequence, for any $\gamma>0$,  we have
$$\EE\left[\sum_{n=1}^{\alpha_k}1_{\{R_{k,n}>\gamma\}}\right] = \tilde{\EE}\left[\sum_{n=1}^{\alpha_k}1_{\{R_{k,n}>\gamma\}}L(E_{k,n})\right],$$
where $L(\cdot)$ is the likelihood function defined in \eqref{eq: LR}.

Suppose we have independently generated $m_1$ regenerative cycles via IS simulation and $m_2$ cycles via naive simulation. 
We use superscript $i$ to denote the $i$-th  cycle and ``$\tilde{~~}$" to denote  samples generated by IS simulation.  
Our IS estimator for tail probability of the steady-state response time is given by 
\begin{equation}\label{eq:tail_estimator}
	\hat{P}_{m_1,m_2}\left(R_{k,\infty}>\gamma\right)
	=\frac{m_1^{-1}\sum_{i=1}^{m_1}\sum_{n=1}^{\tilde{\alpha}_k^i} L(\tilde{E}_{k,n}^i) 1_{\{\tilde{R}_{k,n}^i>\gamma\}}}{m_2^{-1}\sum_{i=1}^{m_2}\alpha^i_k}.
\end{equation}

\subsection{Cross-Entropy Method for Searching Importance Measure}\label{subsec: CE}
Theoretically, the optimal IS distribution \cite{asmussen2007stochastic} for the numerator $\EE\left[\sum_{n=1}^{\alpha_k} 1_{\{R_{k,n}>\gamma\}}\right]$ is 
$$\frac{\sum_{n=1}^{\alpha_k} 1_{\{R_{k,n}>\gamma \}} \cdot h(\alpha)}{ \EE\left[\sum_{n=1}^{\alpha_k} 1_{\{R_{k,n}>\gamma\}}\right]},$$
where $h(\alpha)$ is the density of a regenerative cycle under the original distribution. Note that the optimal IS distribution is not  practical as it involves the unknown constant $\EE\left[\sum_{n=1}^{\alpha_k} 1_{\{R_{k,n}>\gamma\}}\right]$. Still,  if we can choose a distribution close to the optimal IS distribution, it is likely that using it as IS distribution could efficiently  reduce the variance. 
Based on this idea,  cross-entropy method \cite{rubinstein2004cross} uses KL divergence to measure the distance between two distributions and chooses IS distribution by solving the following optimization problem:
\begin{equation}\label{eq:CE}
	\max_{\bm{f}'} \EE\left[   \sum_{n=1}^{\alpha_k} 1_{\{R_{k,n}>\gamma\}}  \cdot  \log L'(\alpha)   \right],
\end{equation}
where $\bm{f}'=({f'}^{A}_{1}, \ldots, {f'}^{A}_{K}, {f'}^{S}_{1}, \ldots, {f'}^{S}_{K})$  are the densities of the inter-arrival and service times under the IS distributions, and  $L'(t)$ is the likelihood ratio with respect to the IS distribution $\bm{f}'$ as defined in \eqref{eq: LR}.

The objective function \eqref{eq:CE} can be estimated by  simulation data, and we can use the IS distribution $\tilde{f}$ for variance reduction.  Suppose we use $N$ simulated regenerative cycles to estimate the objective function, \eqref{eq:CE} is  approximated by

\begin{equation}\label{eq:CEISsample}
	\max_{\bm{f}'}\frac{1}{N}\sum_{i=1}^{N}   \left( \sum_{n=1}^{\tilde{\alpha}_k^i}1_{\{\tilde{R}_{k,n}^i>\gamma\}}L(\tilde{E}_{k,n}^i) \cdot \log L'(\tilde{\alpha}^i) \right).
\end{equation}
For computational efficiency,  $N$ should be much smaller than the number of cycles $m_1$ of \eqref{eq:tail_estimator}. In our setting, we set $\bm{f}'$ as exponential distributions with rates $(\lambda^{\prime}_1,\ldots, \lambda^{\prime}_K,  \mu^{\prime}_1,\ldots, \mu^{\prime}_K)$, which belong to the same parametric family as the original distributions.  Then, problem \eqref{eq:CEISsample} has a closed-form solution:
\begin{equation}\label{eq:lambda}
	{\lambda}_{l}^{\prime} = \frac{ \sum_{i=1}^{M}  \left( \sum_{n=1}^{\tilde{\alpha}_k^i}1_{\{\tilde{R}_{k,n}^i>\gamma\}}L(\tilde{E}_{k,n}^i)  \cdot    {\tilde N_l^{A,i}(\tilde\tau_k^i \wedge \tilde\alpha^i)}  \right) }{ \sum_{i=1}^{M}  \left( \sum_{n=1}^{\tilde{\alpha}_k^i}1_{\{\tilde{R}_{k,n}^i>\gamma\}}L(\tilde{E}_{k,n}^i) \right) \left(\sum_{n=1}^{\tilde N_l^{A,i}(\tilde \tau_k^i \wedge \tilde \alpha^i)} \tilde{A}_{l,n}^i + \tilde{H}_l^{A,i} \right) }, \quad l=1,\ldots, K
\end{equation}
and 
\begin{equation}\label{eq:mu}
	{\mu}_{l}^{\prime} = \begin{cases}
		\frac{ \sum_{i=1}^{M}  \left( \sum_{n=1}^{\tilde{\alpha}_k^i}1_{\{\tilde{R}_{k,n}^i>\gamma\}}L(\tilde{E}_{k,n}^i) \right)     {\tilde N_l^{A,i}(\tilde\tau_k^i \wedge \tilde\alpha^i)}   }{ \sum_{i=1}^{M}  \left( \sum_{n=1}^{\tilde{\alpha}_k^i}1_{\{\tilde{R}_{k,n}^i>\gamma\}}L(\tilde{E}_{k,n}^i) \right) \left(\sum_{n=1}^{\tilde N_l^{A,i}(\tilde\tau_k^i \wedge \tilde\alpha^i)} \tilde{S}_{l,n}^i \right) },& l=1,\ldots, k,\\
		\frac{ \sum_{i=1}^{M}  \left( \sum_{n=1}^{\tilde{\alpha}_k^i}1_{\{\tilde{R}_{k,n}^i>\gamma\}}L(\tilde{E}_{k,n}^i) \right)    {\tilde N_l^{S,i}(\tilde\tau_k^i \wedge \tilde\alpha^i) }    }{ \sum_{i=1}^{M}  \left( \sum_{n=1}^{\tilde{\alpha}_k^i}1_{\{\tilde{R}_{k,n}^i>\gamma\}}L(\tilde{E}_{k,n}^i) \right) \left(\sum_{n=1}^{\tilde N_l^{S,i}(\tilde\tau_k^i \wedge \tilde\alpha^i)} \tilde S_{l,n} \right) },& l=k+1,\ldots, K.
	\end{cases}	 
\end{equation}

However, the cross-entropy method is not suitable for the rare event simulation directly, as the objective \eqref{eq:CEISsample} is $0$ with a high probability.
To address the issue, we use the adaptive cross-entropy method to select good IS distributions by adopting a two-stage procedure where both the level $\gamma$ and the IS distributions are iteratively updated \cite{rubinstein2004cross}.
Specifically, each iteration consists of two steps. In the first step, the algorithm samples $N$ regenerative cycles from the IS distribution with rates $\lambda_{l,t-1}$, $\mu_{l,t-1}$ for $l=1,\ldots,K$, and updates the level $\gamma_t$  as the $1-\rho$ sample quantile of the maximum response times within each cycle. 
In the second step, the algorithm solves \eqref{eq:CEISsample} with $\gamma =\gamma_t$ and updates the IS rates $\lambda_{l,t}=\lambda_l'$, $\mu_{l,t}=\mu_l'$ for $l=1,\ldots,K$. The iteration terminates when $\gamma_t$ exceeds the target level $\gamma$. 
More details are given in the Algorithm \ref{alg}.

\subsection{From Tail Probability to SLA}\label{subsec: complete algorithm}
Finally, we explain how to estimate $Q_k(p)$ using the simulated regenerative cycles under IS distribution. Let $G(\gamma)=P(R_{\target,\infty}>\gamma)$ be the true tail probability of the response time.  By definition, 
\begin{equation*}
	Q_k(p)=\inf\{\gamma:P(R_{\target,\infty}\le \gamma)\ge p\}=\inf\{\gamma: G(\gamma)<1-p\}.
\end{equation*}
According to \eqref{eq:tail_estimator}, for any $\gamma\in [0,\gamma_{\max}]$, we can estimate 
$\hat{G}(\gamma)=\hat{P}_{m_1,m_2}(R_{k,\infty}>\gamma).$ Then, a natural estimator for $Q_k(p)$ is 
\begin{equation*}
	\hat{Q}^k_{m_1,m_2}(p)=\inf\left\{\gamma: \hat{P}_{m_1,m_2}(R_{k,\infty}>\gamma)<1-p 
	\right\}.
\end{equation*}
Equation \eqref{eq:quantile_estimator} can be evaluated in the following way. Suppose the $m_1$ regenerative cycles generated by IS contain in total $\tilde{\beta}_{m_1}$ jobs of type $k$. 
We indexed their response time and likelihood ratio by $\tilde{R}_{k,n}$ and $L_{k,n}$ denote the response time of class-$k$ job $n$ and the corresponding likelihood ratio, respectively. Sort $\tilde R_{k,n}$ in ascending order, thereby forming the ordered samples $(\tilde{R}_{k,(1)}, \ldots,\tilde R_{k,(\tilde \beta_{m_1})})$.
Then,
\begin{equation}\label{eq:quantile_estimator}
	\hat{Q}^k_{m_1,m_2}(p)=\tilde R_{k,\left(n_p\right)}, \text{ where }n_p =\min\left\{n: \sum_{i=n}^{\tilde \beta_{m_1}} L_{\target,(i)} \le \frac{(1-p)m_1}{m_2} \sum_{i=1}^{m_2}\alpha^i_k\right\}.
\end{equation}
Now we are ready to present the complete SLA simulation algorithm.\\

\begin{algorithm}[H]\label{alg}
	\caption{SLA Simulation Algorithm}
	\textbf{Input:}
target class $k$, the number of cycles $m_1$ and $m_2$, and  $\gamma_{\max}>Q_k(p)$.
	\\
	\textbf{1. Denominator Estimation.} Generate $m_2$ cycles by direct Monte Carlo. Obtain $\alpha_k^i$ for $i=1,...,m_1$.
	\\
	\textbf{2. Call CE algorithm for Searching IS measure.} 
	Input parameter $N\geq 1$ and $\rho\in(0,1)$. Initialize $\lambda_{l,0}=\lambda_l,\mu_{l,0}=\mu_l$ for $l=1,2,...,K$.\\
	\For{$t=1,\ldots$}{
		\textbf{Adaptive update of $\gamma_{t}$.} Generate $N$ cycles under the distribution with rates $\lambda_{l,t-1}$, $\mu_{l,t-1}$, and compute the sample $(1-\rho)$-quantile ${\gamma}_{t}$, according to
		$\left\{\max_{n=1,\ldots,\alpha_k^i } R_{k,n}:l=1,\ldots,N\right\}.$\\
		Update $\gamma_t = \min\{\gamma_t, \gamma_{\max}\}$.\\
		\textbf{Adaptive update of IS distributions.} 
		Plugin the data from the $N$ cycles into \eqref{eq:lambda} and \eqref{eq:mu} to obtain $\lambda_{l,t}$ and $\mu_{l,t}$.
		\\
		\If{The adaptive level $\gamma_t$ has reached $\gamma_{\max}$, i.e. $\gamma_t=\gamma_{\max}$}{
			Set the IS measure with rates $\tilde\lambda_{l}=\lambda_{l,t}$, $\tilde\mu_{l}=\mu_{l,t}$ for $l=1,\ldots,K$, and	break the loop.
		}
	} 
	\textbf{3. IS for Tail Probability Estimation.} Sample $m_1$ cycles under the switching change of measure with $\tilde{f}^A_l\sim \exp(\tilde{\lambda}_l)$ and $\tilde{f}^S_l\sim\exp(\tilde{\mu}_l)$. For all $\tilde{\beta}_{m_1}$ jobs of class $k$, record their response time $\tilde{R}_{k,n}$ and likelihood ratio $L_{k,n}$ computed according to \eqref{eq: LR}.\\

	\textbf{4. Quantile Estimation.} Compute $\hat{Q}^k_{m_1,m_2}(p)$ according to \eqref{eq:quantile_estimator}.\\
	\textbf{Output:} $\hat{Q}^k_{m_1,m_2}(p)$.
\end{algorithm}
\vspace{2ex}
\begin{remark}
	The assumption on the availability of $\gamma_{\max}$ can actually be relaxed. Given the CLT result in Section \ref{sec: CI}, we can heuristically test whether $\gamma_t$ obtained in each iteration in CE is an upper bound for $Q_k(p)$.  Then, we terminate the CE iteration when the test result is yes for the first time and use $\gamma_t$ as  $\gamma_{\max}$ in the following IS step. In the numerical experiments, we do not assume any knowledge on $\gamma_{\max}$  and apply this heuristic method.
\end{remark}

\section{CONFIDENCE INTERVAL}\label{sec: CI}
Our next step is to construct the confidence interval for the estimator $\hat{Q}^k_{m_1,m_2}(p)$. To do this, we first need to establish the central limit theorem.  
We consider the case where $m_1=m_2=m$. For any given class $k$ and SLA level $p$,  write $\hat{Q}_m(p)=\hat{Q}^k_{m,m}(p)$.
For $\gamma>0$ and $i=1,2,...,m$, define $$\tilde{Y}_{k, i}(\gamma)=\sum_{n=1}^{\tilde{\alpha}_k^i} L_{k,n} 1_{\{\tilde R_{k,n}>\gamma\}},\text{ and }Z_{k, i}(\gamma)=\tilde{Y}_{k, i} - P(R_{k,\infty}>\gamma)\alpha_{k}^i.$$ 
Denote  $\sigma^2(\gamma)=\operatorname{var}(Z_{k,i}(\gamma)) $. Following the approach in \citeN{iglehart1976simulating}, we establish the following CLT result for our SLA estimator $\hat{Q}_m^p$.

\begin{theorem}
	Suppose $m_1=m_2=m$. For any $k$ and $p$,   denoted  $\hat{Q}_m(p)=\hat{Q}^k_{m,m}(p)$.  Let $F$ be the CDF of $R_{k, \infty}$. 
	 Suppose for all $\gamma$ in some neighborhood of $Q(p)$,  $F^{\prime \prime}(\gamma)$ exists, $\left|F^{\prime \prime}(\gamma)\right| \leq M<\infty$ and $E[|Z_{\target,1}(\gamma)|^{2+\varepsilon}]\leq M<\infty$ for some constants $\varepsilon, M>0$ independent of $\gamma$. Then, as $m \rightarrow \infty$,
   \begin{equation*}
   	\frac{\sqrt{m}(\hat{Q}_{m}(p)-Q(p))}{\sigma(Q(p)) /(E[\alpha_{k}] F^{\prime}(Q(p)))} \Rightarrow N(0,1).
   \end{equation*}  
\end{theorem}

\begin{proof}
    Define $\hat{F}_{m}(\cdot)=1-\hat{P}_{m,m}(R_{k, \infty}>\cdot)$. Then, $\hat{F}_m(\gamma)$ is right-continuous and non-decreasing for all $\gamma>0$. Therefore,
    \begin{equation*}
        P\left(\hat{Q}_{m}(p)\le \gamma\right)=P\left(p\le \hat{F}_{m}(\gamma)\right).
    \end{equation*}
Let $$a_{m}(\gamma)=Q(p)+\frac{\gamma \sigma(Q(p))}{\sqrt{m}E\left[\alpha_\target\right] F^{\prime}(Q(p))}.$$ 
Then,
$$
\begin{aligned}
A_{m}(\gamma) & \triangleq P\left(\frac{\sqrt{m}(\hat{Q}_{m}(p)-Q(p))}{\sigma(Q(p)) / (E[\alpha_\target] F^{\prime}(Q(p)))} \leq \gamma\right) \\
&=P\left( \hat{Q}_m(p) \le a_m(\gamma) \right) = P\left( p\leq\hat{F}_m(a_m(\gamma)) \right)=P\left( \hat{P}_m(R_{\target, \infty}>a_m(\gamma)) \le 1-p\right)\\
&=P\left(\sum_{i=1}^{m} Z_{\target, i}(a_m(\gamma)) \le  \sum_{i=1}^{m} \alpha_{\target, i}  \cdot \left(F\left(a_{m}(\gamma)\right)-p\right)\right).
\end{aligned}
$$
Following the regularity conditions on $F(\cdot)$, we can expand $F(\cdot)$ around  $Q(p)$ as
$$
F\left(a_{m}(\gamma)\right)=p+\frac{\gamma \sigma(Q(p))}{ \sqrt{m}E[\alpha_\target] }+O\left(\frac{1}{m}\right).
$$
Then,
$$
A_{m}(\gamma)=P\left(\frac{1}{\sqrt{m} \sigma({Q(p)})} \sum_{i=1}^{m} Z_{\target,i}\left(a_{m}(\gamma)\right) \le \frac{ \gamma}{m E[\alpha_{\target}]}\left( \sum_{i=1}^{m} \alpha_{\target, i}\right)\left(1+O\left(\frac{1}{m}\right)\right)\right)
$$
By strong law of large number, $\left( \sum_{i=1}^{m} \alpha_{\target, i}\right)/(mE[\alpha_\target])\to 1$ almost surely as $m\to\infty$. Therefore,
$$\frac{ \gamma}{m E[\alpha_\target]}\left( \sum_{i=1}^{m} \alpha_{\target, i}\right)\left(1+O\left(\frac{1}{m}\right)\right)\to \gamma, $$
almost surely as $m\to\infty$.
On the other hand, as $E[|Z_{\target,1}(\gamma)|^{2+\varepsilon}]\leq M<\infty$ for all $\gamma$ in the neighborhood of $Q(p)$, the Liapunov condition holds and therefore  by Corollary 9.8.1 of \citeN{resnick2019probability} , as $m \rightarrow \infty$,
$$
\frac{1}{\sqrt[]{m} \sigma{(a_{m}(\gamma))}} \sum_{i=1}^{m} Z_{\target,i}\left(a_{m}(\gamma)\right) \Rightarrow N(0,1). 
$$
Finally, by the continuity of $\sigma(\cdot)$ around $Q(p)$  and the continuous mapping theorem, we have
$$
\frac{1}{\sqrt[]{m} \sigma{(Q(p))}} \sum_{i=1}^{m} Z_{\target,i}\left(a_{m}(\gamma)\right) \Rightarrow N(0,1).
$$
Therefore, we can conclude
\begin{equation*}
    A_m(\gamma)\rightarrow \Phi(\gamma), \quad\text{ as }m\to \infty.
\end{equation*}
where $\Phi$ is the standard normal distribution function.
\end{proof}

\vspace{2ex}
The CLT result indicates that $\hat{Q}_m(p)$ is a consistent estimator for $Q(p)$. In addition, given the CLT result, we can use batching method  \cite{hsieh2002confidence}  to construct confidence intervals. 
Suppose $m = cr$ so that we divide $m$ cycles into $r$ batches of $c$ cycles each. 
Let $\hat{Q}_{c,i}(p)$ denote the sample quantile computed from \eqref{eq:quantile_estimator} using data in the $i$-th batch. Then, the variance is estimated by  $$\hat\sigma^{2}_m = \frac{1}{r-1} \sum_{i=1}^{r}\left(\hat{Q}_{c,i}(p)-\hat{Q}_m(p)\right)^{2},$$
and the corresponding $95\%$ confidence interval is
\begin{equation}\label{eq: CI}
	CI_m=\left[\hat{Q}_m(p)- \frac{1.96\hat{\sigma}_m}{\sqrt{r}},\hat{Q}_m(p)+\frac{1.96\hat{\sigma}_m}{\sqrt{r}}\right].
\end{equation}

\section{NUMERICAL RESULTS}\label{sec: numerical}
In the numerical experiments, we first illustrate the difference in the choice of IS distribution for preemptive and non-preemptive systems with an example of 2 priority queues.  Then, we implement Algorithm \ref{alg} using 2 priority queues, and compare its performance with two benchmark simulation algorithms. Finally, we apply our algorithm to a system with 8 priority queues as used in Huawei CloudEngine 12800 and 12800E to estimate the SLA performance metrics.

\subsection{Non-preemptive versus Preemptive}\label{subsec: non-preemptive v.s. preemptive}
In this part, we illustrate the difference between non-preemptive and preemptive systems using a simple example of 2 priority queues. The system parameter $(\lambda_1,\lambda_2, \mu_1,\mu_2)= (0.2, 0.4, 2, 1)$, i.e. high priority jobs arrive less frequent and are served faster, which is consistent to our observation in real computing system as described in Section \ref{subsec: 8 queues}.  In \citeN{setayeshgar2012large}, the author derives an asymptotic optimal  IS distribution via LDP analysis for a preemptive system, and we refer to this distribution as LDP in the rest of the section. In Table \ref{tab: premeptive vs nonpreemptive} below, we compare the conditional distribution of the system under LDP, and the IS distribution learned by CE procedure as in Algorithm \ref{alg} (referred as CE), and the original distribution, given that a class-1 job experiences a long delay ($>\hat{Q}_1(0.999)\approx6.913$).

\begin{table}[htbp]
	\centering
	\caption{Conditional probabilities of LDP and CE importance distribution compared to that of the original distribution estimated by 1000,000 regenerative cycles. A cycle is said effective if in this cycle there exists at least one class-1 job experiencing long delay.}\label{tab: premeptive vs nonpreemptive}
	\begin{tabular}{|c|c|c|c|}
		\hline
	&\textbf{\# effective}&\multicolumn{2}{|c|}{\textbf{Conditional probability}}\\
	\textbf{Method}&\textbf{cycle}&after a class-1 job& after a class-2 job\\
	\hline
	Naive&514&0.128&0.872\\
	LDP&546652&0.454&0.544\\
	CE&548942&0.107&0.892\\ 
	\hline
	\end{tabular}
\end{table}

In Table \ref{tab: premeptive vs nonpreemptive},   we report the estimated conditional probabilities that a long-delayed job is blocked by a class-1 and that by a class-2 job right before it. According to the results, we can see that the conditional probabilities estimated by data simulated by LDP are quite different from those by the original distribution, while those by CE are close to those by the original distribution. The comparison results explain why IS results derived from preemptive systems can not be directly applied to non-preemptive systems.
\subsection{Performance Test}\label{subsec: algorithm validation}
We  implement Algorithm \ref{alg} with a system with 2 priority queues for performance test. In particular, we compare Algorithm \ref{alg} with two benchmarks:
\begin{itemize}
	\item Naive: regenerative simulation using original distribution.
	\item LDP: regenerative simulation using IS distribution derived for preemptive system via LDP analysis as given in \citeN{setayeshgar2012large}.
\end{itemize}
We consider a 2 priority queue with parameter $\mu_1=\mu_2$ so that the true value of $Q_k(p)$ can be computed via Laplace transformation \cite{kella1985waiting}. Given the true value of $Q_k(p)$, we use mean square error (MSE) and CI coverage rate to measure the performance of a simulation algorithm. For the purpose of efficiency comparison, we fixed the number of regenerative cycles in each simulation round and estimated MSE and CI coverage rate by 100 rounds of simulation. For Algorithm \ref{alg}, we set $N=10,000$ and $\rho =10\%$ for the CE step. The simulation results for 0.999-quantile of the steady-state response times  are reported in Table \ref{tab:2queue 999}. The results show that our algorithm is more efficient than the two benchmarks as it obtains a significantly smaller MSE and a higher CI coverage rate, given the same number of regenerative cycles. 

Compared to Naive and LDP, Algorithm 1 has an extra numerical routine to optimize importance distribution by CE, which will introduce extra computation cost in addition to simulating the regenerative cycles. Although we did not have a  theoretic guarantee on the convergence speed of CE, in the numerical experiments with 2 priority-queues, we find CE obtains good IS parameters in just a few iterations and thus does not add much extra computation time. For example, the IS parameter used by Algorithm \ref{alg} as reported in Table \ref{tab:2queue 999} is obtained by CE in 8 iterations that take 5.94 seconds per iteration.

\begin{table}[htbp]
	\centering
	\caption{Results for 0.999-quantile estimation of steady-state response time in a 2 priority queue with $(\lambda_1,\lambda_2,\mu_1,\mu_2)=(0.1,0.2,1,1)$. Sample mean, MSE and CI coverage rate are estimated based on 100 rounds of simulation.}
	\begin{tabular}{|c|cccccccccc|}
		\hline 
		\textbf{class} &
		\textbf{Method} &
		  \textbf{\#cycle}  &
		  $\bm{\tilde{\lambda}_1}$&$\bm{\tilde{\lambda}_2}$&$\bm{\tilde{\mu}_1}$&$\bm{\tilde{\mu}_2}$&
		  {$\bm{Q(p)}$}&
		  \textbf{\begin{tabular}[c]{@{}c@{}}sample\\ mean\end{tabular}}
		  &\textbf{MSE} &\textbf{\begin{tabular}[c]{@{}c@{}} coverage\\ rate of CI\end{tabular}}
		    \\ \hline
			\multirow{6}{*}{high}&Naive&10,000& -&-&-&-&8.524&8.598&0.535&65\%\\
		&Naive&100,000& -&-&-&-&8.524&8.477&0.046&65\%\\
		\cline{2-11} 
		&LDP&10,000& 0.333&0.500&0.300&0.300&8.524&8.451&0.0394&63\%\\
		&LDP&100,000& 0.333&0.500&0.300&0.300&8.524&8.500&0.0170&31\%\\
			\cline{2-11}
			&Alg 1&10,000&0.330&0.224&0.234&0.238&8.524&8.527&0.0035&91\%\\
			&Alg 1&100,000&0.330&0.224&0.234&0.238&8.524&8.522&0.0005&86\%\\
		\hline
		\multirow{6}{*}{low}&Naive&10,000& -&-&-&-&11.541&11.618&1.560&61\%\\
		&Naive&100,000& -&-&-&-&11.541&11.524&0.127&54\%\\
		\cline{2-11} 
			&LDP&10,000& 0.333&0.500&0.300&0.300&11.541&11.429&0.305&35\%\\
			&LDP&100,000& 0.333&0.500&0.300&0.300&11.541&11.488&0.0414&35\%\\
			\cline{2-11} 
		&Alg 1&10,000&0.219&0.435&0.341&0.339&11.541&11.496&0.0227&78\%\\
		&Alg 1&100,000&0.219&0.435&0.341&0.339&11.541&11.542&0.0004&71\%\\
		\hline	
	\end{tabular}
	\label{tab:2queue 999}
\end{table}

We also test the simulation algorithms for estimating more extreme quantiles and the results are reported in Table \ref{tab:2queue extreme}. We find the difference in the performance measures between Algorithm \ref{alg} and the benchmarks becomes more significant as $p$ gets closer to 1.
\begin{table}[htbp]
	\centering
	\caption{Results for quantile estimation of steady-state response time in a 2 priority queue with $(\lambda_1,\lambda_2,\mu_1,\mu_2)=(0.1,0.2,1,1)$ with $p_1=1-10^{-5}$ and $p_2=1-10^{-10}$. Sample mean, MSE and CI coverage rate are estimated based on 100 rounds of simulation with 100,000 regenerative cycles in each round. }
	\label{tab:2queue extreme}
	\begin{tabular}{|c|c|cccc|cccc|}
	\hline
	  \textbf{\begin{tabular}[c]{@{}c@{}}class\end{tabular}} &
	  \textbf{Method} &
	  \bm{$Q(p_1)$} &
	  \textbf{\begin{tabular}[c]{@{}c@{}}sample \\ mean\end{tabular}} &
	  \textbf{MSE} &
	  \textbf{\begin{tabular}[c]{@{}c@{}}coverage  \\ rate of CI\end{tabular}} &
	  \bm{$Q(p_2)$} &
	  \textbf{\begin{tabular}[c]{@{}c@{}}sample \\ mean\end{tabular}} &
	  \textbf{MSE} &
	  \textbf{\begin{tabular}[c]{@{}c@{}}coverage  \\ rate of CI\end{tabular}} \\ \hline
	\multirow{3}{*}{high} & Naive & 13.809 & 13.177 & 1.8733 & 8\%  & 26.753 &  -    &  -    &  -   \\
                      & LDP   & 13.809 & 13.778 & 0.0153 & 65\% & 26.753 & 26.591 & 0.1331 & 47\% \\
                      & Alg 1 & 13.809 & 13.804 & 0.0008 & 82\% & 26.753 & 26.739 & 0.0035 & 71\% \\ \hline
\multirow{3}{*}{low}  & Naive & 20.637 & 20.251 & 5.5800 & 26\% & 44.211 &  -   &   -   &  -  \\
                      & LDP   & 20.637 & 20.411 & 0.4468 & 26\% & 44.211 & 42.430 & 6.0641 & 4\%  \\
                      & Alg 1 & 20.637 & 20.617 & 0.0166 & 54\% & 44.211 & 44.208 & 0.0270 & 72\% \\ \hline
	\end{tabular}

	\end{table}

\subsection{Application to SLA Estimation}\label{subsec: 8 queues}
We consider a priority queueing model used in Huawei CloudEngine as illustrated by Figure \ref{fig: 8 queues}.
\begin{figure}[ht]
	\begin{center}
			\includegraphics[scale=0.5]{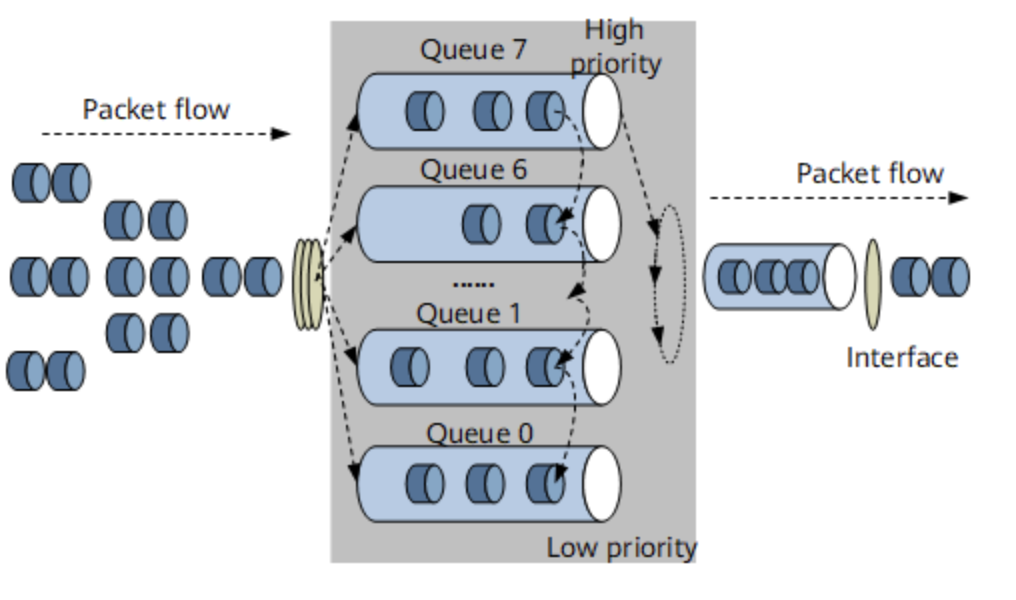}
	\end{center}
	\caption{Illustration of a network device with Priority Queueing scheduling.}\label{fig: 8 queues}
\end{figure}
 
 There are eight queues with indexes ranging from 0 to 7 on each interface in the outbound direction of a switch. The priorities of queues 7 to 0 are in the descending order defined by the Internet Engineering Task Force (IETF), i.e., class selector 7 (CS7), CS6, expedited forwarding (EF), assured forwarding 4 (AF4), AF3, AF2, AF1, and best effort (BE), respectively. Current Internet practice suggests CS7 and CS6 for the network control and signaling, EF for telephony, AF4 to AF1 for multimedia streaming, and BE for flows without bandwidth assurance such as email and telnet services. According to the flow characteristics, the model parameters are set as in Table \ref{tab: 8 queue parameters}. We then apply Algorithm \ref{alg} with number of regenerative cycles $m_1=m_2=100,000$, $N=10,000$ and $\rho = 0.1$ for the CE step. The simulation results and running times are reported in Table \ref{tab:8class 39} and Table \ref{tab:8class 59}. 
\begin{table}[ht]
	\caption{Model parameters  of 8 priority queues.}\label{tab: 8 queue parameters}
	\begin{center}
			\begin{tabular}{|c|cccccccc|}
			\hline
			Queue index&7&6&5&4&3&2&1&0\\
			\hline
			Service class&CS7&CS6&EF&AF4&AF3&AF2&AF1&BE\\
			\hline
			\text{mean packet size (byte)}&100&100&200&1400&1400&1400&1400&1400\\
			\text{mean arrival rate (Mbps)}&10&10&10&400&200&100&100&450\\
			\hline
		\end{tabular}
	\end{center}

\end{table}

Our goal is to estimate $Q_k(p)$ for $k=0,1,...,7$ and $p=0.999$ and $0.99999$. To control the estimation accuracy, we use a pre-simulation to estimate the CLT variance via batching and then choose the number of regenerative cycles $m$  so that the relative error (defined as $\hat{\sigma}_m/\hat{Q}_m(p)$) is less than 0.1\%.  In this case, as $\mu_k$ are unequal, the true value of $Q_k(p)$ is not available. So to verify the simulation estimation $\hat{Q}_k(p)$, we run independent IS simulation using 100,000 regenerative cycles to estimate $P(R_{k,\infty}>\hat{Q}_k(p))$. If our estimation is accurate, the estimated probability should be close to $p$.  The simulation is implemented in Python and runs on a 16-cores server with Intel Cascade Lake 3.0GHz CPU. The simulation results and computation times (including pre-simulation and all steps in Algorithm \ref{alg}) are reported in Table \ref{tab:8class 39} and Table \ref{tab:8class 59} below.  It is notable that when the tail probability $1-p$ shrinks by 100 times, the total computation time  of our method only doubles from $12327$s to $26875$s to achieve the same relative precision level.

\begin{table}[ht]
	\centering
	\caption{ Simulation results and running time  for 0.999-quantile estimation of steady-state response time of 8 priority classes. The unit of response time is microsecond ($\mu$s).}
	\begin{tabular}{|c|ccc|cc|c|}
	\hline
	\textbf{class} &
	\textbf{\begin{tabular}[c]{@{}c@{}}sample\\  quantile\end{tabular}} &
	\textbf{CI} &
	\textbf{\#cycle} &
	\textbf{probability} &
	\textbf{CI} &
	\textbf{\begin{tabular}[c]{@{}c@{}}running\\ time (s)\end{tabular}} \\ \hline
	7 & 22.855 & (22.811, 22.900) & 2836384 & 0.001028 & (0.000946, 0.001111) & 1409 \\
	6 & 22.880 & (22.835, 22.925) & 3074112 & 0.001020 & (0.000949, 0.001091) & 1738 \\
	5 & 23.260 & (23.215, 23.305) & 3060256 & 0.001041 & (0.000923, 0.001159) & 1774 \\
	4 & 33.970 & (33.906, 34.034) & 817056  & 0.001003 & (0.000962, 0.001044) & 470  \\
	3 & 45.013 & (44.926, 45.099) & 1987216 & 0.000962 & (0.000911, 0.001014) & 1101 \\
	2 & 53.259 & (53.155, 53.364) & 3847392 & 0.001059 & (0.000983, 0.001135) & 2434 \\
	1 & 59.139 & (59.028, 59.251) & 4140432 & 0.001018 & (0.000943, 0.001093) & 2592 \\
	0 & 74.354 & (74.208, 74.500) & 1604864 & 0.001006 & (0.000964, 0.001048) & 809  \\ \hline
	\end{tabular}
	\label{tab:8class 39}
	\end{table}

\begin{table}[ht]
	\centering
	\caption{ Simulation results and running time   for 0.99999-quantile estimation of steady-state response time of 8 priority classes. The unit of response time is microsecond ($\mu$s).}
	\begin{tabular}{|c|ccc|cc|c|}
		\hline
		\textbf{class} &
		\textbf{\begin{tabular}[c]{@{}c@{}}sample\\  quantile\end{tabular}} &
		\textbf{CI} &
		\textbf{\#cycle} &
		\textbf{probability} &
		\textbf{CI} &
		\textbf{\begin{tabular}[c]{@{}c@{}}running\\ time (s)\end{tabular}} \\ \hline
	7 & 39.818  & (39.748, 39.888)   & 4155440 & 0.0000093 & (8.554e-06,   1.012e-05) & 3321 \\
	6 & 39.664  & (39.587, 39.742)   & 3614400 & 0.0000101 & (9.194e-06, 1.107e-05)   & 2801 \\
	5 & 40.480  & (40.400, 40.559)   & 4278560 & 0.0000091 & (8.142e-06, 1.001e-05)   & 3054 \\
	4 & 54.354  & (54.247, 54.460)   & 268624  & 0.0000110 & (9.499e-06, 1.247e-05)   & 268  \\
	3 & 81.637  & (81.478, 81.797)   & 1950080 & 0.0000096 & (8.794e-06, 1.036e-05)   & 1717 \\
	2 & 100.922 & (100.725, 101.120) & 5303456 & 0.0000089 & (7.536e-06, 1.023e-05)   & 7127 \\
	1 & 113.862 & (113.639, 114.085) & 6569632 & 0.0000099 & (7.976e-06, 1.191e-05)   & 7758 \\
	0 & 139.734 & (139.460, 140.008  & 1217664 & 0.0000106 & (9.981e-06, 1.130e-05)   & 828  \\ \hline
	\end{tabular}
	\label{tab:8class 59}
	\end{table}

\section{CONCLUSION}\label{sec:}

In this article, we proposed a new simulation algorithm to estimate tail quantiles of the steady-state sojourn time in non-preemptive priority queues. Our algorithm is designed based on regenerative simulation, and importance sampling is used to improve efficiency. The importance distribution is optimized by the cross-entropy method. The confidence interval is also constructed. The numerical experiments show that our algorithm obtains significant improvement compared to the benchmarks. Future directions of interest include extending the current framework to queues with other service protocols such as WFQ and DRR, commonly used in computing and telecommunication systems.

\appendix

\footnotesize

\printbibliography

\end{document}